\documentclass{amsart}
\usepackage{amsmath}
\usepackage{amssymb}
\usepackage{amsfonts}

\newtheorem{theorem}{Theorem}
\theoremstyle{plain}

\newtheorem{corollary}{Corollary}

\newtheorem{definition}{Definition}

\newtheorem{lemma}{Lemma}

\newtheorem{remark}{Remark}

\numberwithin{equation}{section}
\input{tcilatex}

\begin{document}
\title[]{$n$-tuplet fixed point of multivalued mappings via measure of
noncompactness}
\author{Derya Sekman}
\address{Department of Mathematics, Faculty of Arts and Sciences, Ahi Evran
University, 40100 K\i r\c{s}ehir, Turkey}
\email{deryasekman@gmail.com}
\urladdr{}
\thanks{}
\author{Nour El Houda Bouzara}
\address{Faculty of Mathematics, University of Science and Technology Houari
Boumedi\`{e}ne, Bab-Ezzouar, 16111 Algies, Algeria}
\email{bzr.nour@gmail.com}
\urladdr{}
\thanks{}
\author{Vatan Karakaya}
\address{Department of Mathematical Engineering, Y\i ld\i z Technical
University, 34210 Istanbul, Turkey}
\email{vkkaya@yahoo.com}
\urladdr{}
\dedicatory{}
\thanks{}

\begin{abstract}
In this work, by using measure of noncompactness, some results on the
existence of $n$-tuplet fixed points for multivalued contraction mappings
are proved. As an application, the existence of solution for a system of
integral inclusions is studied.
\end{abstract}

\maketitle

\section{\protect\bigskip Introduction}

Fixed point theory is an important area in mathematics which is closely
related to real world problems and has many applications in different fields
of science. It has played an important role in solving problems of
uniqueness and existence of the nonlinear analysis, topology and geometry.
In the literature, it has applications in many sciences like engineering 
\cite{fle}, economy \cite{ok}, optimization \cite{ceng}, game theory \cite%
{bor} and medicine \cite{chen}. Brouwer \cite{bro} proved fixed point
theorem on $n$-dimensional space. Banach \cite{ban} established contraction
principle which supplied find fixed point theorem using contraction mapping.
Later, Schauder \cite{sch} proved an extension of the Brouwer's fixed point
theorem to spaces of infinite dimension under compactness condition.

For fixed point theory of single valued mappings, one can consult in \cite%
{b1,b2,Schauder}. The case of multivalued mappings compared to single valued
mappings direct application to the real world problems \cite{n1,karlin}. The
literature is important in view of its. Along with that Kakutani \cite%
{kakutani} extended Brouwer's fixed point theorem to multivalued mappings.
Afterward, Nadler \cite{nadler} extended Banach contraction principle from
single valued mappings to multivalued mappings using Hausdorff metric.
Later, classes of Nadler's fixed point theorem was extended and generalized
for various multivalued mappings in \cite{reich,mizoguchi}.

Measure of noncompactness has played a fundamental role in the study of
single valued and multivalued mappings, especially, in metric and
topological fixed point theory. It is very useful tool to guarantees the
existence of fixed point. The measure of noncompactness was defined and
studied by Kuratowski \cite{kurotowski}. Darbo \cite{darbo1} used this
measure to generalize both Schauder's fixed point theorem and Banach's
contraction principle for condensing operators. Recently, measure of
noncompactness has been used in differential equations, integral equations,
nonlinear equations as given in \cite{darbo,ag1,ag2}.

Partially ordered metric spaces are very important in fixed point theory. By
using two basic concepts, Guo and Lakshmikantham \cite{guo} first gave some
existence theorems of the coupled fixed point for both continuous and
discontinuous operators, and then they offered some applications to the
initial value problems of ordinary differential equations with discontinuous
right-hand sides. Bhaskar and Lakshmikantham \cite{bhaskar} introduced
coupled fixed point and established in a some coupled fixed point theorems
in a partially ordered metric space. By a similar idea, Berinde and Borcut 
\cite{borcut} established a tripled \ fixed point for nonlinear mapping in
partially ordered complete metric spaces. Ert\"{u}rk and Karakaya \cite%
{erturk} introduced the concept of $n$-tuplet fixed point and studied
existence and uniqueness of fixed point of contractive type mappings in
partially ordered metric spaces. Moreover, by using the condensing
operators, Aghajani et al. \cite{article1} presented some results on the
existence of coupled fixed point Karakaya et al. \cite{vatan} gave some
results concerning the existence of tripled fixed point via measure of
noncompactness. For different models of measure of noncompactness in \cite{A}%
.

The existence of fixed point for various contractive mappings has been
studied by many authors under different conditions. The concept of \ coupled
fixed point for multivalued mappings was introduced by Samet and Vetro \cite%
{samet} and they presented coupled fixed point theorem for multivalued
nonlinear contraction mappings in a partially ordered metric space. Rao et
al. \cite{Rao} obtained a tripled coincidence fixed point theorems for
multivalued mappings in a partially ordered metric space.

In this paper, by using condensing operator, we investigate $n$-tuplet fixed
point of multivalued mappings on a Banach space. Finally, we also give an
application of our result to solve a system of integral inclusions.

\section{Preliminaries}

Throughout this paper, $E$ is a Banach space and $P(E)$ (or $2^{E}$) is the
set of all subsets of $E$. \ We denote the set 
\begin{equation*}
P_{k}(E)=\left\{ X\subset E,\text{ }X\text{ is nonempty and has a property }%
k\right\} \text{.}
\end{equation*}
So, $P_{rcp}(E),P_{cl,bd}(E),P_{cl,cv}(E)$ will denote the classes of all
relatively compact, closed-bounded and closed-convex subsets of $E,$
respectively.

A mapping $T:E\rightarrow P_{k}(E)$ is called a multivalued mapping or set
valued mapping on $E$ into $E$. A point $x\in E$ is called a fixed point of $%
T$ if $x\in Tx$.

\begin{definition}[see; \protect\cite{Dhage}]
A mapping $\mu :P_{cl,bd}(X)\rightarrow 
\mathbb{R}
^{+}$ is called a measure of noncompactness if it satisfies the following
conditions:
\end{definition}

$($M$_{1})$ $\varnothing \neq \mu ^{-1}(0)\subset P_{rcp}(X)$,

$($M$_{2})$ $\mu (\bar{A})=\mu (A)$, where $\bar{A}$ denotes the closure of $%
A$,

$($M$_{3})$ $\mu (conv$ $A)=$ $\mu (A)$, where $conv$ $A$ denotes the convex
hull of $\ A$,

$($M$_{4})$ $\mu $ is nondecreasing,

$($M$_{5})$ If $\left\{ A_{n}\right\} $ is a decreasing sequence of sets in $%
P_{cl,bd}(X)$ satisfying $\lim\limits_{n\rightarrow \infty }\mu (A_{n})=0$,
then the intersection%
\begin{equation*}
A_{\infty }=\underset{n=1}{\overset{\infty }{\cap }}A_{n}
\end{equation*}%
is nonempty.

If $($M$_{4})$ holds, then $A_{\infty }\in P_{rcp}(X).$ For this, let $%
\lim\limits_{n\rightarrow \infty }\mu (A_{n})=0.$ As $A_{\infty }\subseteq
A_{n}$ for each $n=0,1,2,...;$ by the monotonicity of $\mu ,$ we obtain 
\begin{equation*}
\mu (A_{\infty })\leq \lim\limits_{n\rightarrow \infty }\mu (A_{n})=0.
\end{equation*}%
So, by $($M$_{1})$, we get that $A_{\infty }$ is nonempty and$\ A_{\infty
}\in P_{rcp}(X)$.

\begin{theorem}[see; \protect\cite{Schauder}]
Let $X$ be a closed and convex subset of a Banach space $E$. Then every
compact, continuous map $T:X\rightarrow X$\ has at least one fixed point.
\end{theorem}

\begin{theorem}[see; \protect\cite{darbo}]
Let $X$ be a nonempty, bounded, closed and convex subset of a Banach space $%
E $ and let\ $T:X\rightarrow X$ be a continuous mapping. Suppose that there
exists a constant $k\in \left[ 0,1\right) $ such that%
\begin{equation*}
\mu (T(X))\leq k\mu (X)
\end{equation*}%
for any subset $X$ of $E$, then $T$ has a fixed point.
\end{theorem}

\begin{definition}[see; \protect\cite{Dhage}]
A multivalued mapping\ $T:E\rightarrow P_{cl,bd}(E)$ is called $D$-$set$-$%
Lipschitz$ if there exists a continuous nondecreasing function $\varphi :%
\mathbb{R}^{+}\rightarrow \mathbb{R}^{+}$ such that%
\begin{equation*}
\mu (T(X))\leq \varphi (\mu (X))
\end{equation*}%
for all $X\in P_{cl,bd}(E)$ with $T(X)\in P_{cl,bd}(E$), where $\varphi
(0)=0 $. Generally, we call the function $\varphi $ to be a$\ D$-$function$
of $T$ on $E$.

When $\varphi (r)=kr$, $k>0$, $T$ is called a $k$-$set$-$Lipschitz$ mapping
and if $k<1$, then $T$ is called a $k$-$set$-$contraction$ on $E$.
\end{definition}

If $\varphi (r)<r$ for $r>0$, then $T$ is called a $nonlinear$ $D$-$set$-$%
contraction$ on E.

\begin{lemma}[see; \protect\cite{asym}]
If $\varphi $ is a $D$-$function$ with $\varphi (r)<r$ for $r>0$, then%
\begin{equation*}
\lim\limits_{n\rightarrow \infty }\varphi ^{n}(t)=0
\end{equation*}%
for all $t\in \left[ 0,\infty \right) $.
\end{lemma}

\begin{theorem}[see; \protect\cite{Dhage}]
\label{cont}Let $X$ be a nonempty, bounded, closed and convex subset of a
Banach space $E$ and let \ $T:X\rightarrow P_{cl,cv}(X)$ be a closed and $%
nonlinear$ $D$-$set$-$contraction$. Then $T$ has a fixed point.
\end{theorem}

\bigskip As a consequence of Theorem \ref{cont}, we obtain a fixed point
theorem of Darbo (\cite{darbo}) type for linear set-contractions.

\begin{corollary}[see; \protect\cite{Dhage}]
\label{Dhage}Let $X$ be a bounded, closed and convex subset of a Banach
space $E$ and let \ $T:X\rightarrow P_{cl,cv}(X)$ be a closed and $k$-$set$-$%
contraction$. Then $T$ has a fixed point.
\end{corollary}

\begin{definition}[see; \protect\cite{itoh}]
Let $X$ be a topological space, $2^{X}$ the family of all subsets of $X$ and 
$T$ be a mapping of $X$ into $2^{X}$ such that $Tx$ is nonempty, for all $%
x\in X$. Then the mapping $T$ is called upper semicontinuous if for each
closed subset $C$ of $X$, 
\begin{equation*}
T^{-1}(C)=\left\{ x\in X:Tx\cap C\neq \varnothing \right\}
\end{equation*}%
is closed.
\end{definition}

\begin{definition}[see; \protect\cite{Dhage}]
A mapping $\mu :P_{k}(E)\rightarrow 
\mathbb{R}
^{+}$ is called nondecreasing if $A,B\in P_{k}(E)$ are any two sets with $%
A\subseteq B$, then $\mu (A)\leq \mu (B)$, where $\subseteq $ is order
relation of inclusion in $P_{k}(E)$.
\end{definition}

\begin{lemma}[see; \protect\cite{Thi}]
\label{graph}Let X be a Banach space and $F$ be a Caratheodory multivalued
mapping. Let $\Phi :L^{1}\left( H;X\right) \rightarrow C\left( H;X\right) $
be linear continuous mapping. Then, 
\begin{eqnarray*}
\Phi \circ S_{F}:C\left( H;X\right) &\rightarrow &\mathcal{P}_{cl;cv}\left(
C\left( H;X\right) \right) \\
u &\rightarrow &\left( \Phi \circ S_{F}\right) u:=\Phi \left( S_{F}\left(
u\right) \right) ,
\end{eqnarray*}%
is a closed graph operator in $C\left( H;X\right) \times C\left( H;X\right) $%
.
\end{lemma}

\begin{lemma}[see; \protect\cite{KaObZec}]
\label{equimnclem}
\end{lemma}

$1.$ Let $A\subseteq C\left( H;X\right) $ be bounded. Then $\mu \left(
A\left( t\right) \right) \leqslant \mu \left( A\right) $ for all $t\in H$,
where $A\left( t\right) =\left\{ y\left( t\right) ,y\in A\right\} \subset X$%
. Furthermore, if $A$ is equicontinuous on $H$, then $\mu \left( A\left(
t\right) \right) $ is continuous on $H$ and $\mu \left( A\right) =\sup
\left\{ \mu \left( A\left( t\right) \right) ,t\in H\right\} .$

$2.$ If $A\subset C\left( H;X\right) $ is bounded and equicontinuous, then%
\begin{equation*}
\mu \left( \int_{0}^{t}A\left( s\right) ds\right) \leqslant \int_{0}^{t}\mu
\left( A\left( s\right) \right) ds,
\end{equation*}%
for all $t\in H,$ where $\int_{0}^{t}A\left( s\right) ds=\left\{
\int_{0}^{t}x\left( s\right) ds:x\in A\right\} .$

\section{\protect\bigskip $n$-tuplet Fixed Point Theorems and Some Related
Results}

In this section, we investigate $n$-tuplet fixed point property of a
multivalued mapping and give some applications for special cases $n=2$, that
is, coupled fixed point.

\begin{definition}
\label{tuplet}Let $X$ be a nonempty set and $G:X^{n}\rightarrow P(X)$ be a
given mapping. An element $(x_{1},x_{2},x_{3},...,x_{n})\in X^{n}$ is called
an $n$-tuplet fixed point of $G$ if%
\begin{eqnarray*}
x_{1} &\in &G(x_{1},x_{2},x_{3},...,x_{n}) \\
x_{2} &\in &G(x_{2},x_{3},...,x_{n},x_{1}) \\
&&\vdots \\
x_{n} &\in &G(x_{n},x_{1},x_{2},...,x_{n-1})\text{.}
\end{eqnarray*}
\end{definition}

\begin{remark}
If we take as special cases $n=2$ and $n=3$ in Definition \ref{tuplet},
respectively, we get coupled fixed point (see; \cite{samet}) and tripled
fixed point (see; \cite{Rao}).
\end{remark}

\begin{theorem}
\label{Akm}(see; \cite{akmerov}) Let $\mu _{1},\mu _{2}...,\mu _{n}$ be
measures of noncompactness in Banach spaces $E_{1},E_{2}...,E_{n}$
respectively. Suppose that the function $F:[0,\infty )^{n}\rightarrow
\lbrack 0,\infty )$ is convex and $F(x_{1},x_{2},...,x_{n})=0$ if and only
if $x_{i}=0$ for $i=1,2,...,n$. Then 
\begin{equation*}
\widetilde{\mu }(X)=F(\mu _{1}(X_{1}),\mu _{2}(X_{2})...,\mu _{n}(X_{n})),
\end{equation*}%
defines a measure of noncompactness in $E_{1}\times E_{2}\times ...\times
E_{n}$ where $X_{i}$ denotes the natural projection of $X$ onto $E_{i}$, for 
$i=1,2,...,n.$
\end{theorem}

\begin{remark}
\label{rem}Notice that by taking%
\begin{equation*}
F\left( x_{1},x_{2},x_{3},...,x_{n}\right) =\max \left\{
x_{1},x_{2},x_{3},...,x_{n}\right\} ,
\end{equation*}%
or%
\begin{equation*}
F\left( x_{1},x_{2},x_{3},...,x_{n}\right) =x_{1}+x_{2}+x_{3}+...+x_{n},
\end{equation*}%
for any $\left( x_{1},x_{2},x_{3},...,x_{n}\right) \in \left[ 0,\infty
\right) ^{n}$, the conditions of Theorem \ref{Akm} are satisfied. Therefore,%
\begin{equation*}
\widetilde{\mu }\left( X\right) :=\max \left( \mu \left( X_{1}\right) ,\mu
\left( X_{2}\right) ,...,\mu \left( X_{n}\right) \right) ,
\end{equation*}%
or 
\begin{equation*}
\widetilde{\mu }\left( X\right) :=\mu \left( X_{1}\right) +\mu \left(
X_{2}\right) +\cdots +\mu \left( X_{n}\right) ,
\end{equation*}%
\ defines measures of noncompactness in the space $E^{n},$ where $X_{i}$, $%
i=1,2,...,n$ are the natural projections of $X$ on $E_{i}$.
\end{remark}

We now give an important theorem for existence of fixed point of multivalued
mapping under measure of noncompactness condition.

\begin{theorem}
\label{teo11}Let $X$ be a nonempty, bounded, closed and convex subset of a
Banach space $E$\ and let $\mu $ \ be an arbitrary measure of noncompactness
in it. Let $\varphi :\mathbb{R}^{+}\rightarrow \mathbb{R}^{+}$ be a
nondecreasing and upper semicontinuous function such that $\varphi \left(
r\right) <r$ for all $r>0.$ Suppose that $G:X_{1}\times X_{2}\times \cdots
\times X_{n}\rightarrow P_{cl,cv}(X)$ is continuous multivalued operator
satisfying%
\begin{equation*}
\mu \left( G\left( X_{1}\times X_{2}\times \cdots \times X_{n}\right)
\right) \leqslant \varphi \left( \frac{\mu \left( X_{1}\right) +\mu \left(
X_{2}\right) +\cdots +\mu \left( X_{n}\right) }{n}\right)
\end{equation*}%
for all $X_{1},X_{2},...,X_{n}\subset X$. Then $G$ has at least one $n$%
-tuplet fixed point.
\end{theorem}

\begin{proof}
As in Remark \ref{rem}, we define the measure of noncompactness $\tilde{\mu}$
by 
\begin{equation*}
\widetilde{\mu }\left( X\right) :=\mu \left( X_{1}\right) +\mu \left(
X_{2}\right) +\cdots +\mu \left( X_{n}\right) .
\end{equation*}%
Define the mapping $\tilde{G}(X):=G\left( X_{1}\times X_{2}\times \cdots
\times X_{n}\right) $. We prove that $\widetilde{G}$ satisfies all the
conditions of Theorem \ref{cont}. Then

Clearly, 
\begin{eqnarray*}
\widetilde{\mu }\left( \widetilde{G}\left( X\right) \right) &=&\widetilde{%
\mu }\left( G\left( X_{1}\times X_{2}\times \cdots \times X_{n}\right)
\right) \\
&=&\widetilde{\mu }\left(
G(x_{1},x_{2},x_{3},...,x_{n}),G(x_{2},x_{3},...,x_{n},x_{1}),...,G(x_{n},x_{1},x_{2},...,x_{n-1})\right)
\\
&=&\mu \left( G(x_{1},x_{2},x_{3},...,x_{n}))+\mu
(G(x_{2},x_{3},...,x_{n},x_{1}))+\cdots +\mu
(G(x_{n},x_{1},x_{2},...,x_{n-1})\right) \\
&\leqslant &\varphi \left( \frac{\mu \left( X_{1}\right) +\mu \left(
X_{2}\right) +\cdots +\mu \left( X_{n}\right) }{n}\right) +\varphi \left( 
\frac{\mu \left( X_{2}\right) +\mu \left( X_{3}\right) +\cdots +\mu \left(
X_{1}\right) }{n}\right) \\
&&+\cdots +\varphi \left( \frac{\mu \left( X_{n}\right) +\mu \left(
X_{1}\right) +\cdots +\mu \left( X_{n-1}\right) }{n}\right) \\
&=&n\varphi \left( \frac{\mu \left( X_{1}\right) +\mu \left( X_{2}\right)
+\cdots +\mu \left( X_{n}\right) }{n}\right) \text{.}
\end{eqnarray*}%
Now,%
\begin{equation*}
\frac{1}{n}\widetilde{\mu }\left( \widetilde{G}\left( X\right) \right) \leq
\varphi \left( \frac{\mu \left( X_{1}\right) +\mu \left( X_{2}\right)
+\cdots +\mu \left( X_{n}\right) }{n}\right)
\end{equation*}%
and taking $\widetilde{\mu }^{\prime }=\frac{1}{n}\widetilde{\mu },$ we get 
\begin{equation*}
\widetilde{\mu }^{\prime }\left( \widetilde{G}\left( X\right) \right)
\leqslant \varphi \left( \widetilde{\mu }^{\prime }\left( X\right) \right) 
\text{.}
\end{equation*}%
Also, $\widetilde{\mu }^{\prime }$ is a measure of noncompactness. Thus, by
Theorem \ref{cont}, we obtain that $G$ has at least one $n$-tuplet fixed
point.
\end{proof}

\begin{remark}
If we take $\mu $ measure of noncompactness in Theorem \ref{teo11} as 
\begin{equation*}
\widetilde{\mu }\left( X\right) :=\max \left( \mu \left( X_{1}\right) ,\mu
\left( X_{2}\right) ,...,\mu \left( X_{n}\right) \right) .
\end{equation*}%
We can obtain the same result.
\end{remark}

\section{Application To Inclusions Systems}

The multivalued fixed point theorem of this paper has some nice applications
to differential and integral systems of inclusions as an example we study
the solvability of a system of differential inclusions.

Consider the following differential system%
\begin{equation}
\left\{ 
\begin{array}{c}
x^{\prime }(t)\in A(t)x(t)+G(t,x\left( t\right) ,y\left( t\right) ),\ \ t\in 
\left[ 0,b\right] \\ 
y^{\prime }(t)\in A(t)y(t)+F(t,y\left( t\right) ,x\left( t\right) ),\ \ t\in 
\left[ 0,b\right]%
\end{array}%
\right.  \label{evsys}
\end{equation}%
with%
\begin{equation}
x(0)=\varphi (x,y),\text{ }y(0)=\varphi (y,x)\text{\ \ }  \label{evcond}
\end{equation}%
where $G$ is an upper Caratheodory multimap, $\varphi :C\left( \left[ 0,b%
\right] ,X\right) \rightarrow X$ is a given multivalued function. $\left\{
A\left( t\right) :t\in \left[ 0,b\right] \right\} $ is a family of linear
closed unbounded operators on $X$ with domain $D(A(t))$ independent of $t$
that generate $\Delta $ an evolution system of operators $\left\{ U\left(
t,s\right) :t,s\in \Delta \right\} $ with $\Delta =\left\{ \left( t,s\right)
\in \left[ 0,b\right] \times \left[ 0,b\right] :0\leqslant s\leqslant
t\leqslant b\right\} $.

Define the set%
\begin{equation*}
S_{G}\left( x,y\right) =\left\{ g\in L^{1}\left( \left[ 0,b\right] ,X\right)
:g\left( t\right) \in G\left( t,x\left( t\right) ,y\left( t\right) \right)
\right\} .
\end{equation*}

\begin{definition}
A family $\{U(t,s)\}_{(t,s)\in \Delta }$ of bounded linear operators $%
U(t,s):X\rightarrow X$ where $(t,s)\in \Delta :=\{(t,s)\in J\times
J:0\leqslant s\leqslant t<+\infty \}$ for $J=\left[ 0,b\right] $ is called
an evolution system if the following properties are satisfied,

\begin{enumerate}
\item $U(t,t)=I$ where $I$ is the identity operator in $X$ and $U(t,s)\
U(s,\tau )=U(t,\tau )$ for $0\leqslant \tau \leqslant s\leqslant t<+\infty $,

\item The mapping $(t,s)\rightarrow U(t,s)\ y$ is strongly continuous, that
is, there exists a constant $M>0$ such that 
\begin{equation*}
\Vert U(t,s)\Vert \leqslant M\ \ \text{for any }(t,s)\in \Delta .
\end{equation*}
\end{enumerate}
\end{definition}

An evolution system $U(t,s)$ is said to be compact if $U(t,s)$ is compact
for any $t$-$s>0$. $U(t,s)$ is said to be equicontinuous if $\left\{
U(t,s)x:x\in M\right\} $ is equicontinuous at $0\leqslant s<t\leqslant b$
for any bounded subset $B\subset X$. Clearly, if $U(t,s)$ is a compact
evolution system, it must be equicontinuous. The converse is not necessarily
true.

More details on evolution systems and their properties could be found in the
books of Ahmed \cite{Ah}, Engel and Nagel \cite{EnNa} and Pazy \cite{pazy}.

\begin{definition}
We say that the couple\ $\left( x\left( t\right) ,y(t)\right) \in C\left( %
\left[ 0,b\right] ,X\right) \times C\left( \left[ 0,b\right] ,X\right) $ is
a mild solution of the evolution system $\left( \text{\ref{evsys}}\right)
-\left( \text{\ref{evcond}}\right) $ if it satisfies the following integral
system 
\begin{equation}
\left\{ 
\begin{array}{c}
x(t)=U(t,0)\ \varphi (x,y)+\int_{0}^{t}U(t,s)\ g(s)\ ds\text{ \ \ for }g\in
S_{G}\left( x,y\right) \\ 
y(t)=U(t,0)\ \varphi (y,x)+\int_{0}^{t}U(t,s)\ g(s)\ ds\text{ \ \ for }g\in
S_{G}\left( y,x\right)%
\end{array}%
,\right.  \label{evmild}
\end{equation}%
for all $t\in \left[ 0,b\right] $.
\end{definition}

\begin{theorem}
Assume the following hypotheses
\end{theorem}

\begin{itemize}
\item[$\left( H1\right) $] $\left\{ A\left( t\right) :t\in J\right\} $ is a
family of linear operators. $A\left( t\right) :D\left( A\right) \subset
X\rightarrow X$ generates an equicontinuous evolution system $\left\{
U\left( t,s\right) :\left( t,s\right) \in \Delta \right\} $ and%
\begin{equation*}
\left\vert U\left( t,s\right) \right\vert \leqslant M.
\end{equation*}
\end{itemize}

\begin{itemize}
\item[$(H2)$] The multifunction $G:J\times C(\left[ 0,b\right] \times
X\times X)\longrightarrow \mathcal{P}_{cl,cv}(X)$ is an upper Carath\'{e}%
odory with respect to $x$ and $y$ and $\varphi :C(J;X)\rightarrow X$ is
compact and 
\begin{equation*}
\mu \left( G\left( t,W\times W\right) \right) <k\mu \left( \frac{W\times W}{2%
}\right) \ \text{for any }t\in J.
\end{equation*}

\item[$\left( H3\right) $] There exists a constant $r>0$ such that 
\begin{equation*}
M\left[ \left\Vert \varphi \left( x,y\right) \right\Vert +\left\{ \left\Vert
g\left( t\right) \right\Vert _{1}:g\in S_{G}\left( x,y\right) ,x\in
A_{0}\right\} \right] \leqslant r
\end{equation*}

and 
\begin{equation*}
\ M\left[ \left\Vert \varphi \left( y,x\right) \right\Vert +\left\{
\left\Vert g\left( t\right) \right\Vert _{1}:g\in S_{G}\left( y,x\right)
,y\in A_{0}\right\} \right] \leqslant r
\end{equation*}
where, $A_{0}=\left\{ z\in C(J;X):\left\Vert z\left( t\right) \right\Vert
\leqslant r\text{ for all }t\in J\right\} $ hold. Then the non local system $%
\left( \text{\ref{evsys}}\right) -\left( \text{\ref{evcond}}\right) $ has at
least one mild solution in the space $C\left( J,X\right) $.
\end{itemize}

\textbf{Proof.} To solve problem given in $\left( \text{\ref{evsys}}\right)
-\left( \text{\ref{evcond}}\right) ,$ we transform it into the following
fixed point problem.

Consider the multivalued operator $N:C(\left[ 0,b\right] ;X;X)\rightarrow 
\mathcal{P}(C(\left[ 0,b\right] ;X))$ defined by, 
\begin{equation*}
N(x,y)=\left\{ h\in C(J;X):h(t)=U(t,0)\varphi (x,y)+\int_{0}^{t}U(t,s)\
g(s)\ ds,\text{ with }g\in S_{G}\left( x,y\right) \right\} .
\end{equation*}

Clearly, coupled fixed points of the operator $N$ are mild solutions of
system $\left( \text{\ref{evcond}}\right) -$ $($\ref{evmild}$)$.

Obviously, for each $y\in C(\left[ 0,b\right] ;X)$, the set $S_{G}\left(
x,y\right) $ is nonempty since, by $(H2)$, $G$ has a measurable selection
(see \cite{CaVa}).

Let show that $N$ has a coupled fixed point. For that, we need to verify all
the conditions of Theorem \ref{teo11}.

Let $A_{0}=\left\{ z\in C(\left[ 0,b\right] ;X):\left\Vert z\left( t\right)
\right\Vert \leqslant r\text{ for all }t\in \left[ 0,b\right] \right\} $. We
notice that $A_{0}$ is closed, bounded and convex.

To show that $N\left( A_{0}\times A_{0}\right) \subseteq A_{0}$, we need
first to prove that the family 
\begin{equation*}
\left\{ \int_{0}^{t}U(t,s)\ f(s)\ ds\ :f\in S_{F}\left( y\right) \text{ and }%
y\in A_{0}\right\}
\end{equation*}%
is equicontinuous for $t\in J$, that is, all the functions are continuous
and they have equal variation over a given neighbourhood.

In view of $\left( H1\right) $ we have that funtions in the set $\left\{
U\left( t,s\right) :\left( t,s\right) \in \Delta \right\} $ are
equicontinuous, i.e, for every $\varepsilon >0$ there exists $\delta >0$
such that $\left\vert t-\tau \right\vert <\delta $ implies $\left\Vert
U\left( t,s\right) -u\left( \tau ,s\right) \right\Vert <\varepsilon $ for
all $U\left( t,s\right) \in \left\{ U\left( t,s\right) :\left( t,s\right)
\in \Delta \right\} .$

Then, given some $\varepsilon >0$ let $\delta =\frac{\varepsilon ^{\prime }}{%
\varepsilon \left\Vert g\right\Vert _{\infty }}$ such that $\left\vert
t-\tau \right\vert <\delta $, we have%
\begin{equation*}
\left\vert \int_{0}^{t}U(t,s)\ g(s)\ ds-\int_{0}^{\tau }U(\tau ,s)gf(s)\
ds\right\vert \leqslant \int_{\tau }^{t}\left\vert U\left( t,s\right)
-U\left( \tau ,s\right) \right\vert \left\vert g\left( s\right) \right\vert
ds.
\end{equation*}%
As $\left\{ U\left( t,s\right) :\left( t,s\right) \in \Delta \right\} $ is
equicontinuous, so we have%
\begin{eqnarray*}
\left\vert \int_{0}^{t}U(t,s)\ g(s)\ ds-\int_{0}^{\tau }U(\tau ,s)\ g(s)\
ds\right\vert &\leqslant &\varepsilon \left\Vert g\right\Vert _{\infty
}\left\vert t-\tau \right\vert \\
&<&\varepsilon \left\Vert g\right\Vert _{\infty }\frac{\varepsilon ^{\prime }%
}{\varepsilon \left\Vert g\right\Vert _{\infty }}=\varepsilon ^{\prime }.
\end{eqnarray*}%
Hence we conclude that $\left\{ \int_{0}^{t}U(t,s)\ g(s)\ ds\ :g\in
S_{G}\left( x,y\right) \text{ and }\left( x,y\right) \in A_{0}\times
A_{0}\right\} $ is equicontinuous for $t\in J$.

Now, we show that $N\left( A_{0}\times A_{0}\right) \subseteq A_{0}$. For $%
t\in J$, we have%
\begin{eqnarray*}
\left\vert h\left( t\right) \right\vert &=&\left\vert U\left( t,0\right)
\varphi (x,y)+\int_{0}^{t}U(t,s)\ g(s)\ ds\right\vert \\
&\leqslant &\left\vert U\left( t,0\right) \varphi (x,y)\right\vert
+\int_{0}^{t}\left\vert U(t,s)\ g(s)\right\vert ds \\
&\leqslant &M\left\Vert \varphi \left( x,y\right) \right\Vert +M\left\Vert
g\right\Vert _{1} \\
&=&M\left[ \left\Vert \varphi \left( x,y\right) \right\Vert +\left\Vert
g\right\Vert _{1}\right] \leqslant r.
\end{eqnarray*}%
Thus $N\left( A_{0}\times A_{0}\right) \subseteq A_{0}$.

Further, it is easy to see that $N$ is convex value.

Now, let us show that $N$ has a closed graph, let $x_{n}\rightarrow x$, $%
y_{n}\rightarrow y$ and $h_{n}\rightarrow h$ such that $h_{n}\left( t\right)
\in N\left( x_{n},y_{n}\right) $ and we show that $h\left( t\right) \in
N\left( x,y\right) .$

Now, there exists a sequence $g_{n}\in S_{G}\left( x_{n},y_{n}\right) $ such
that%
\begin{equation*}
h_{n}\left( t\right) =U\left( t,0\right) \varphi
(x_{n},y_{n})+\int_{0}^{t}U(t,s)g_{n}(s)ds.
\end{equation*}%
Consider the linear operator $\Phi :L^{1}\left( \left[ 0,b\right] ;X\right)
\rightarrow C\left( \left[ 0,b\right] ;X\right) $ defined by 
\begin{equation*}
\Phi f\left( t\right) =\int_{0}^{t}U(t,s)g_{n}(s)ds.
\end{equation*}%
Clearly, $\Phi $ is linear and continuous. So by Lemma \ref{graph}, we get
that $\Phi \circ S_{G}\left( x,y\right) $ is a closed graph operator.
Further, we have%
\begin{equation*}
h_{n}\left( t\right) -U\left( t,0\right) \varphi (x_{n},y_{n})\in \Phi \circ
S_{G}\left( x,y\right) .
\end{equation*}%
Since $x_{n}\rightarrow x$, $y_{n}\rightarrow y$ and $h_{n}\rightarrow h,$
therefore%
\begin{equation*}
h\left( t\right) -U\left( t,0\right) \varphi (x,y)\in \Phi \circ S_{G}\left(
x,y\right) .
\end{equation*}%
That is, there exists a function $g\in S_{G}\left( x,y\right) $ such that%
\begin{equation*}
h\left( t\right) =U\left( t,0\right) \varphi (y)+\int_{0}^{t}U(t,s)g(s)ds.
\end{equation*}%
Therefore $N$ has a closed graph, hence $N$ has closed values on $C\left( %
\left[ 0,b\right] \times X\times X,X\right) $.

We know that the family $\left\{ \int_{0}^{t}U(t,s)f(s)ds,f\in S_{F}\left(
W\left( t\right) \right) \right\} $ is equicontinuous, hence by Lemma \ref%
{equimnclem}, we have%
\begin{eqnarray*}
\mu \left( \int_{0}^{t}U(t,s)g(s)ds,\text{ }g\in S_{G}\left( W\left(
t\right) \times W\left( t\right) \right) \right) &\leqslant &\int_{0}^{t}\mu
\left( U(t,s)g(s),\text{ }g\in S_{G}\left( W\left( t\right) \times W\left(
t\right) \right) \right) ds \\
&\leqslant &M\int_{0}^{t}\mu \left( g(s),\text{ }g\in S_{G}\left( W\left(
t\right) \times W\left( t\right) \right) \right) ds \\
&\leqslant &Mt\mu \left( G\left( t,W\left( t\right) \right) \right) .
\end{eqnarray*}%
Therefore%
\begin{eqnarray*}
\mu \left( N\left( W\times W\right) \right) &=&\mu N\left( U\left(
t,0\right) \varphi (W\left( t\right) \times W\left( t\right)
)+\int_{0}^{t}U(t,s)g(s)ds,g\in S_{G}\left( W\left( t\right) \times W\left(
t\right) \right) \right) \\
&\leqslant &\mu \left( U\left( t,0\right) \varphi (W\left( t\right) \times
W\left( t\right) )\right) +\mu \left( \int_{0}^{t}U(t,s)g(s)ds,g\in
S_{G}\left( W\left( t\right) \times W\left( t\right) \right) \right) \\
&\leqslant &M\mu \left( \varphi (W\left( t\right) \times W\left( t\right)
)\right) +Mt\mu \left( G\left( t,W\left( t\right) \times W\left( t\right)
\right) \right) .
\end{eqnarray*}%
In view of $\left( H2\right) $, we get%
\begin{equation*}
\mu \left( N\left( W\times W\right) \right) \leqslant Mbk\mu \left( \frac{%
W\times W}{2}\right) .
\end{equation*}%
Therefore, for $Mbk<1,$ we obtain that $N$ has at least one coupled fixed
point. Hence, the system $\left( \text{\ref{evsys}}\right) -\left( \text{\ref%
{evcond}}\right) $ has at least one solution.

\end{document}